\documentclass[12pt,a4paper]{amsart}
\usepackage{amsfonts}
\usepackage{amsthm}
\usepackage{amsmath}
\usepackage{amscd}
\usepackage[latin2]{inputenc}
\usepackage{t1enc}
\usepackage[mathscr]{eucal}
\usepackage{indentfirst}
\usepackage{graphicx}
\usepackage{graphics}
\usepackage{pict2e}
\usepackage{epic}
\usepackage{color}
\numberwithin{equation}{section}
\usepackage[margin=2.9cm]{geometry}
\usepackage{epstopdf}

\newtheorem{theorem}{Theorem}[section]
\newtheorem{lemma}[theorem]{Lemma}

\theoremstyle{definition}
\newtheorem{definition}[theorem]{Definition}
\newtheorem{example}[theorem]{Example}

\begin{document}
\setcounter{page}{1}
\title{Integral Operator Frames on Hilbert C*-modules}
\author[H. Labrigui, M. Rossafi, A. Touri and Nadia Assila]{Hatim Labrigui$^{1}$, Mohamed Rossafi$^{*2}$, Abdeslam Touri$^{1}$ and Nadia Assila$^{1}$}
\address{$^{1}$Department of Mathematics, Faculty of Sciences, Ibn Tofail University, Kenitra, Morocco}
\email{hlabrigui75@gmail.com; touri.abdo68@gmail.com; nadia.assila@uit.ac.ma}
\address{$^{2}$LASMA Laboratory Department of Mathematics, Faculty of Sciences Dhar El Mahraz, University Sidi Mohamed Ben Abdellah, Fes, Morocco}
\email{\textcolor[rgb]{0.00,0.00,0.84}{rossafimohamed@gmail.com; mohamed.rossafi@usmba.ac.ma}}
\date{
\newline \indent $^{*}$ Corresponding author}


\date{...}
\maketitle

\begin{abstract}
	Introduced by Duffin and Schaefer as a part of their work on nonhamonic fourrier series in 1952, the theory of frames has undergone a very interesting evolution in recent decades following the multiplicity of work carried out in this field. In this work, we introduce a new concept that of integral operator frame for the set of all adjointable operators on a Hilbert C*-modules H and we give some new  propertis relating for some construction of integral operator frame, also we establish some new results. Some illustrative examples are provided to advocate the usability of our results.
\end{abstract}

\section{Introduction and preliminaries}
\quad
In 1952 Duffin and Schaefer \cite{Duf} have introduced the concept of frames in the study of nonharmonic Fourier series. Frames possess many nice properties which make them very useful in wavelet analysis, irregular sampling theory, signal processing and many other fields. The theory of frames has been generalized rapidly and various generalizations of frames in Hilbert spaces and Hilbert $C^{\ast}$-modules \cite{F4}.
Theory of frames have been extended from Hilbert spaces to Hilbert  $C^{\ast}$-module either in the discrete case or in the continuous case  (see \cite{F4, LG, mjpaa,mjpaaa, r4,  moi,r8, r11,r2}).

The aim of this paper is to extend results of Chun-Yan Li and Huai-Xin Cao \cite{Chun}, given in the discrete case for Hilbert spaces to the continuous case for Hilbert $C^{\ast}$-modules.
For this end, we introduce and we study the concept of integral operator frames for $End_{\mathcal{A}}^{\ast}(\mathcal{H})$. Also, some properties are given. 
In what follows, we set $\mathcal{H}$ a separable Hilbert space and $End_{\mathcal{A}}^{\ast}(\mathcal{H}) $ the set of all adjointable operators from Hilbert $C^{\ast}$-modules $\mathcal{H}$ to  $\mathcal{H}$ and $(\Omega,\mu)$ is a measure space with positive measure $\mu$.\\
Let $K,T \in End_{\mathcal{A}}^{\ast}(\mathcal{H})  $, if $TK=I$, then $T$ is called the left inverse of $K$, denoted by $K_{l}^{-1}$.\\
If $KT=I$, then $T$ is called the right inverse of $K$ and we write $K_{r}^{-1}=T$.\\
If $KT=TK=I$, then $T$ and $K$ are inverse of each other.\\
For a separable Hilbert space $H$ and a measurable space $(\Omega,\mu)$, we define,
\begin{equation*}
l^{2}(\Omega,\mathcal{H})=\{x_{\omega} \in \mathcal{H},\quad  \omega \in \Omega,\quad \left\|\int_{\Omega}\langle x_{\omega},x_{\omega}\rangle d\mu(\omega)\right\| < \infty  \}.
\end{equation*}
For any $x=(x_{\omega})_{\omega \in \Omega}$ and $y=(y_{\omega})_{\omega \in \Omega}$, the inner product on $l^{2}(\Omega,H)$ is defined by, 
\begin{equation*}
\langle x,y\rangle = \int_{\Omega}\langle x_{\omega},y_{\omega}\rangle d\mu(\omega).
\end{equation*}
The norme is defined by $\|x\|=\langle x,x\rangle^{\frac{1}{2}}$.

In this section we briefly recall the definitions and basic properties of $C^{\ast}$-algebra, Hilbert $\mathcal{A}$-modules, frame in Hilbert $\mathcal{A}$-modules. For information about frames in Hilbert spaces we refer to \cite{Ch}. Our references for $C^{\ast}$-algebras are \cite{{Dav},{Con}}.\\
For a $C^{\ast}$-algebra $\mathcal{A}$, if $a\in\mathcal{A}$ is positive we write $a\geq 0$ and $\mathcal{A}^{+}$ denotes the set of positive elements of $\mathcal{A}$.
\begin{definition}\cite{Kap}.
	Let $ \mathcal{A} $ be a unital $C^{\ast}$-algebra and $\mathcal{H}$ be a left $ \mathcal{A} $-module, such that the linear structures of $\mathcal{A}$ and $ \mathcal{H} $ are compatible. $\mathcal{H}$ is a pre-Hilbert $\mathcal{A}$-module if $\mathcal{H}$ is equipped with an $\mathcal{A}$-valued inner product $\langle.,.\rangle_{\mathcal{A}} :\mathcal{H}\times\mathcal{H}\rightarrow\mathcal{A}$, such that is sesquilinear, positive definite and respects the module action. In the other words,
	\begin{itemize}
		\item [(i)] $ \langle x,x\rangle_{\mathcal{A}}\geq0 $ for all $ x\in\mathcal{H} $ and $ \langle x,x\rangle_{\mathcal{A}}=0$ if and only if $x=0$.
		\item [(ii)] $\langle ax+y,z\rangle_{\mathcal{A}}=a\langle x,y\rangle_{\mathcal{A}}+\langle y,z\rangle_{\mathcal{A}}$ for all $a\in\mathcal{A}$ and $x,y,z\in\mathcal{H}$.
		\item[(iii)] $ \langle x,y\rangle_{\mathcal{A}}=\langle y,x\rangle_{\mathcal{A}}^{\ast} $ for all $x,y\in\mathcal{H}$.
	\end{itemize}	 
\end{definition}
For $x\in\mathcal{H}, $ we define $||x||=||\langle x,x\rangle||^{\frac{1}{2}}$. If $\mathcal{H}$ is complete with $||.||$, it is called a Hilbert $\mathcal{A}$-module or a Hilbert $C^{\ast}$-module over $\mathcal{A}$. For every $a$ in $C^{\ast}$-algebra $\mathcal{A}$, we have $|a|=(a^{\ast}a)^{\frac{1}{2}}$ and the $\mathcal{A}$-valued norm on $\mathcal{H}$ is defined by $|x|=\langle x, x\rangle^{\frac{1}{2}}$ for $x\in\mathcal{H}$.

Let $\mathcal{H}$ and $\mathcal{K}$ be two Hilbert $\mathcal{A}$-modules, A map $T:\mathcal{H}\rightarrow\mathcal{K}$ is said to be adjointable if there exists a map $T^{\ast}:\mathcal{K}\rightarrow\mathcal{H}$ such that $\langle Tx,y\rangle_{\mathcal{A}}=\langle x,T^{\ast}y\rangle_{\mathcal{A}}$ for all $x\in\mathcal{H}$ and $y\in\mathcal{K}$.

We also reserve the notation $End_{\mathcal{A}}^{\ast}(\mathcal{H},\mathcal{K})$ for the set of all adjointable operators from $\mathcal{H}$ to $\mathcal{K}$ and $End_{\mathcal{A}}^{\ast}(\mathcal{H},\mathcal{H})$ is abbreviated to $End_{\mathcal{A}}^{\ast}(\mathcal{H})$.
\begin{definition}\cite{BA}
	Let $ \mathcal{H} $ be a Hilbert $\mathcal{A}$-module over a unital $C^{\ast}$-algebra. A family $\{x_{i}\}_{i\in I}$ of elements of $\mathcal{H}$ is said to be a frame for $ \mathcal{H} $, if there
	exist two positive constants $A,B$ such that,
	\begin{equation}\label{eq01}
	A\langle x,x\rangle_{\mathcal{A}}\leq\sum_{i\in I}\langle x,x_{i}\rangle_{\mathcal{A}}\langle x_{i},x\rangle_{\mathcal{A}}\leq B\langle x,x\rangle_{\mathcal{A}}, \qquad x\in \mathcal{H}.
	\end{equation}
	The numbers $A$ and $B$ are called lower and upper bounds of the frame, respectively. If $A=B=\lambda$, the frame is called $\lambda$-tight. If $A = B = 1$, it is called a normalized tight frame or a Parseval frame. If only upper inequality of \eqref{eq01} hold, then $\{x_{i}\}_{i\in I}$ is called a Bessel sequence for $\mathcal{H}$.
\end{definition}
In \cite{LG}, L. Gavruta introduced $K$-frames to study atomic systems for operators in Hilbert spaces.
\begin{definition}\cite{Gav} Let $K\in End_{\mathcal{A}}^{\ast}(\mathcal{H})$. A family $\{x_{i}\}_{i\in I}$ of elements in a Hilbert $\mathcal{A}$-module $\mathcal{H}$ over a unital $C^{\ast}$-algebra is a $K$-frame for $ \mathcal{H} $, if there exist two positive constants $A$ and $B$, such that,
	\begin{equation}\label{11}
	A\langle K^{\ast}x,K^{\ast}x\rangle_{\mathcal{A}}\leq\sum_{i\in I}\langle x,x_{i}\rangle_{\mathcal{A}}\langle x_{i},x\rangle_{\mathcal{A}}\leq B\langle x,x\rangle_{\mathcal{A}}, \qquad x\in\mathcal{H}.
	\end{equation}
	The numbers $A$ and $B$ are called lower and upper bounds of the $K$-frame, respectively.
\end{definition}
The following lemmas will be used to prove our mains results
\begin{lemma} \label{l1} \cite{Pas}
	Let $\mathcal{H}$ be a Hilbert $\mathcal{A}$-module. For $T\in End_{\mathcal{A}}^{\ast}(\mathcal{H})$, we have $$\langle Tx,Tx\rangle_{\mathcal{A}}\leq\|T\|^{2}\langle x,x\rangle_{\mathcal{A}},\qquad  \forall x\in\mathcal{H}.$$
\end{lemma}
\begin{lemma} \label{sb} \cite{Ara}.
	Let $\mathcal{H}$ and $\mathcal{K}$ be two Hilbert $\mathcal{A}$-modules and $T\in End_{\mathcal{A}}^{\ast}(\mathcal{H},\mathcal{K})$. Then the following statements are equivalent:
	\begin{itemize}
		\item [(i)] $T$ is surjective.
		\item [(ii)] There are $m,M>0$ such that  $m\|x\|\leq \|Tx\|\leq M\|x\|$, for all $x\in \mathcal{H}$.
		\item [(iii)] There are $m^{'},M^{'}>0$ such that $m^{'}\langle x,x\rangle_{\mathcal{A}} \leq \langle  Tx,Tx\rangle_{\mathcal{A}}\leq M'\langle x,x\rangle_{\mathcal{A}}$, for all $x\in \mathcal{H}$.
	\end{itemize}
\end{lemma}
\begin{lemma} \label{l2} \cite{Zha}
	Let $\mathcal{H}$ be a Hilbert $\mathcal{A}$-module over a $C^{\ast}$-algebra $\mathcal{A}$ and let $T, S$ two operators for $End_{\mathcal{A}}^{\ast}(\mathcal{H})$. If $Rang(S)$ is closed, then the following statements are equivalent:
	\begin{itemize}
		\item [(i)] $Rang(T)\subseteq Rang(S)$.
		\item [(ii)] $ TT^{\ast}\leq \lambda SS^{\ast}$ for some $\lambda>0$.
		\end{itemize}
\end{lemma}
\begin{lemma} \label{l3} \cite{Ali}.
	Let $\mathcal{H}$ and $\mathcal{K}$ be two Hilbert $\mathcal{A}$-modules and $T\in End^{\ast}(\mathcal{H},\mathcal{K})$. Then:
	\begin{itemize}
		\item [(i)] If $T$ is injective and $T$ has closed range, then the adjointable map $T^{\ast}T$ is invertible and $$\|(T^{\ast}T)^{-1}\|^{-1}\leq T^{\ast}T\leq\|T\|^{2}.$$
		\item  [(ii)]	If $T$ is surjective, then the adjointable map $TT^{\ast}$ is invertible and $$\|(TT^{\ast})^{-1}\|^{-1}\leq TT^{\ast}\leq\|T\|^{2}.$$
	\end{itemize}	
\end{lemma}
\begin{lemma} \label{l4} \cite{33}.
	Let $(\Omega,\mu )$ be a measure space, $X$ and $Y$ are two Banach spaces, $\lambda : X\longrightarrow Y$ be a bounded linear operator and $f : \Omega\longrightarrow X$ measurable function; then, 
	\begin{equation*}
	\lambda (\int_{\Omega}fd\mu)=\int_{\Omega}(\lambda f)d\mu.
	\end{equation*}
\end{lemma}
\begin{theorem}\cite{Ch}\label{t0}
	Let $X$ be a Banach space, $U : X \longrightarrow X$ a bounded operator and $\|I-U\|<1$. Then $U$ is invertible.
\end{theorem}
\section{Integral Operator Frames for $End_{\mathcal{A}}^{\ast}(\mathcal{H})$}
\quad

\begin{definition}
	A family of adjointable operators $\{T_{w}\}_{w\in\Omega} \subset End_{\mathcal{A}}^{\ast}(\mathcal{H}) $ on a Hilbert $\mathcal{A}$-module $\mathcal{H}$ over a unital $C^{\ast}$-algebra is said to be an integral operator frame for $End_{\mathcal{A}}^{\ast}(\mathcal{H})$, if there exist two positive constants $A, B > 0$ such that, 
	\begin{equation}\label{eq3}
		A\langle x,x\rangle_{\mathcal{A}} \leq\int_{\Omega}\langle T_{\omega}x,T_{\omega}x\rangle_{\mathcal{A}} d\mu(\omega)\leq B\langle x,x\rangle_{\mathcal{A}} , \quad x\in\mathcal{H}.
	\end{equation}

	If the sum in the middle of \eqref{eq3} is convergent in norm, the integral operator frame is called standard. 

	The numbers $A$ and $B$ are called respectively lower and upper bound of the integral operator frame.\\
	An integral operator frame $\{T_{\omega}\}_{\omega \in \Omega} \subset End_{\mathcal{A}}^{\ast}(\mathcal{H})$ is said to be $A$-tight if there exists a constant $0 < A$ such that,
\begin{equation*}
A\langle x,x\rangle_{\mathcal{A}}=\int_{\Omega} \langle T_{\omega}x,T_{\omega}x\rangle_{\mathcal{A}} d\mu(\omega), \quad x\in \mathcal{H}
\end{equation*}
If $A=1$, it is called a normalized tight integral operator frames or a Parseval integral operator frame.\\ 
If only the right-hand inequality of \eqref{eq3} is satisfied, we call $T :=\{T_{w}\}_{ w\in\Omega}$ the integral operator Bessel family for $\mathcal{H}$ with Bessel bound $B$.\\
\end{definition}
	\begin{example}
		Let $\mathcal{H}$ be the Hilbert space defined by:\\
		$\mathcal{H}=\left\{ A=\left( 
		\begin{array}{ccc}
		a & 0  \\ 
		0 & b 
		\end{array}%
		\right) \text{ / }a,b\in 
		\mathbb{C}
		\right\} $, \\
		It's know that $\mathcal{H}$ is a $C^{\ast}$-algebra and it's a Hilbert $C^{\ast}$-module over it's self.\\
		We define the inner product :
		\[
		\begin{array}{ccc}
		\mathcal{H}\times \mathcal{H} & \rightarrow  & \mathcal{H} \\ 
		(A,B) & \mapsto  & \langle A, B\rangle = \left( 
		\begin{array}{ccc}
		a\bar{a_{1}} & 0  \\ 
		0 & b\bar{b_{1}} 
		\end{array}%
		\right) %
		\end{array}%
		\]\\
		Now, we consider a measure space $(\Omega=\left[ 0,1\right] ,d\lambda ) $ whose $d\lambda $ is a Lebesgue measure  restraint on the interval $\left[ 0,1\right] $.\\
		For all $w \in \left[ 0,1\right] $, we define :
		\begin{align*}
		T_{\omega} : \mathcal{H} &\longrightarrow \mathcal{H}\\
		\left(   \begin{array}{ccc}
		a & 0  \\ 
		0 & b 
		\end{array}
		\right)&\longrightarrow \left(   \begin{array}{ccc}
		wa & 0  \\ 
		0 & \frac{\sqrt{3}wb}{2} 
		\end{array}
		\right)
		\end{align*}
		It is clear that the family $\{T_{w}\}_{w\in \left[ 0,1\right]}$ is a continuous operator on $\mathcal{H}$.\\
		Moreover, we have :
		\[
		\begin{array}{ccc}
			\langle T_{w}A, T_{w}A\rangle = w^{2}\left( 
		\begin{array}{ccc}
		|a|^{2} & 0  \\ 
		0 & \frac{3|b|^{2}}{4} 
		\end{array}%
		\right) %
		\end{array}%
		\]\\
			
		Hence,
		\[
		\begin{array}{ccc}
		\int_{\Omega}\langle T_{w}A, T_{w}A\rangle d\lambda(w) =\int_{[0,1]} w^{2}d\lambda\left( 
		\begin{array}{ccc}
		|a|^{2} & 0  \\ 
		0 & \frac{3|b|^{2}}{4} 
		\end{array}%
		\right)=\left( 
		\begin{array}{ccc}
		\frac{|a|^{2}}{3} & 0  \\ 
		0 & \frac{|b|^{2}}{4} 
		\end{array}%
		\right) %
		\end{array}%
		\]\\
		
		So, 
		\begin{equation*}
		\frac{1}{4} \langle A,A\rangle\leq \int_{\Omega}\langle T_{w}A, T_{w}A\rangle d\lambda(w) \leq \frac{1}{3} \langle A,A\rangle.
		\end{equation*}
		Which shows that  $\{T_{w}\}_{w\in \left[ 0,1\right]}$ is an integral operator frame for $\mathcal{H}$.
	\end{example}	

Let $T=\{T_{\omega}\}_{\omega\in\Omega}$ be an integral operator frame for $End_{\mathcal{A}}^{\ast}(\mathcal{H})$.\\
we define the operator $R_{T}$ by,
\begin{align*}
R_{T}:\mathcal{H}&\longrightarrow l^{2}(\Omega,\mathcal{H})\\
x&\longrightarrow R_{T}x=\{T_{\omega}x\}_{\omega\in\Omega},  
\end{align*}
The operator $R_{T}$ is called the analysis operator of the integral operator frame $\{T_{\omega}\}_{\omega\in\Omega} $.\\
The adjoint of the analysis operator $R_{T}$ is defined by,
\begin{align*}
R_{T}^{\ast}:l^{2}(\Omega,\mathcal{H})&\longrightarrow \mathcal{H}\\
\{x_{\omega}\}_{\omega\in\Omega}&\longrightarrow R_{T}^{\ast}(\{x_{\omega}\}_{\omega\in\Omega})=\int_{\Omega}T_{\omega}^{\ast}x_{\omega}d\mu(\omega).  
\end{align*}
The operator $R_{T}^{\ast}$ is called the synthesis operator of $ \{T_{\omega}\}_{\omega\in\Omega} $.\\
By composing $R_{T}$ and $R_{T}^{\ast}$, the frame operator $S_{T}:\mathcal{H}\rightarrow\mathcal{H}$ for the operator frame $T$ is given by
\begin{equation*}
S_{T}(x)=R_{T}^{\ast}R_{T}x=\int_{\Omega}T_{\omega}^{\ast}T_{\omega}xd\mu(\omega) \qquad x\in \mathcal{H}.
\end{equation*} 
Assume that $\{T_{\omega}\}_{w\in  \Omega} $ is an integral operator frame for $End_{\mathcal{A}}^{\ast}(\mathcal{H})$ and $R_{T}$, $R^{\ast}_{T}$ are the analysis integral operator frame and synthesis operator frame of $T$ respectively.
We say that $S_{T}$ is the frame operator associeted to the integral operator frame.

\begin{theorem}\label{t2}
	Assume that $S_{T}$ is the integral frame operator of an integral operator frame $\{T_{\omega}\}_{\omega \in \Omega} $ for $ End_{\mathcal{A}}^{\ast}(\mathcal{H})$ with bounds $A,B$. Then $S_{T}$ is selfadjoint, positive and invertible operator on $\mathcal{H}$. Moreover, we have a reconstruction formula : 
	\begin{equation}\label{fr}
	x=\int_{\Omega}S^{-1}_{T}T^{\ast}_{\omega}T_{\omega}xd\mu(\omega) \qquad x\in \mathcal{H}
	\end{equation}
\end{theorem}
\begin{proof}
	
	Let $\{T_{\omega}\}_{\omega \in \Omega} $ be an integral operator frame with bounds $A,B$. Then for all $x\in \mathcal{H}$, we have:
	\begin{equation*}
	\langle Ax,x\rangle_{\mathcal{A}} =A\langle x,x\rangle_{\mathcal{A}}\\
	\leq \int_{\Omega}\langle T_{w}x,T_{w}x\rangle_{\mathcal{A}} d\mu(\omega)=\langle S_{T}x,x\rangle_{\mathcal{A}} \leq B\langle  x,x\rangle_{\mathcal{A}}=\langle  Bx,x\rangle_{\mathcal{A}}.
	\end{equation*}
	This shows that : $A.I\leq S_{T} \leq B.I$. Then
	\begin{equation*}
	0 \leq I-B^{-1}S_{T} \leq \frac{B-A}{B}.I,
	\end{equation*} 
	and consequently:
	\begin{equation*}
	\|I-B^{-1}S_{T} \|=\underset{x \in \mathcal{H}, \|x\|=1}{\sup}\|\langle(I-B^{-1}S_{T})x,x\rangle_{\mathcal{A}} \|\leq \frac{B-A}{B}<1.
	\end{equation*}
	By Theorem \ref{t0}, we shows that  $S_{T}$ is invertible.\\
	It's clear that $S_{T}$ is a positive operator.\\
	Further, for any $x\in \mathcal{H}$, we have :
	\begin{equation}
	x=S^{-1}_{T}S_{T}x=S^{-1}_{T}\int_{\Omega}T^{\ast}_{\omega}T_{\omega}xd\mu(\omega)=\int_{\Omega}S^{-1}_{T}T^{\ast}_{\omega}T_{\omega}xd\mu(\omega)
	\end{equation} 
\end{proof}
\begin{theorem}\label{spp}
	Let $T= \{T_{w}\}_{ w\in \Omega}\subset  End_{\mathcal{A}}^{\ast}(\mathcal{H})$ be  an integral operator Bessel family. Then $T$ is an integral operator frame for $End_{\mathcal{A}}^{\ast}(\mathcal{H})$ if and only if there exist a positive constants $A$ and $B$ such that : 
	\begin{equation} \label{12.3}
	A\|x\|^{2} \leq\| \int_{\Omega}\langle T_{\omega}x,T_{\omega}x\rangle_{\mathcal{A}} d\mu(\omega)\| \leq B\|x\|^{2},  \qquad   x\in \mathcal{H}.
	\end{equation}
\end{theorem}
\begin{proof}
	Suppose that \eqref{12.3} holds, we have,\\
	\begin{equation}\label{2.5}
	\langle S^{\frac{1}{2}}_{T}x,S^{\frac{1}{2}}_{T}x\rangle_{\mathcal{A}}=\langle S_{T}x,x\rangle_{\mathcal{A}} =\int_{\Omega}\langle T_{\omega}x,T_{\omega}x\rangle_{\mathcal{A}} d\mu(w).
	\end{equation}
	Using \eqref{12.3}, for all $x\in \mathcal{H}$ we have,
	\begin{equation*}
	\sqrt{A}\|x\| \leq \|S^{\frac{1}{2}}_{T}x\|\leq \sqrt{B}\|x\|, 
	\end{equation*}
	by Lemma \ref{sb}, there exist $m, M > 0 $ such that :
	\begin{equation*}
	m\langle x,x\rangle_{\mathcal{A}} \leq \langle S^{\frac{1}{2}}_{T}x,S^{\frac{1}{2}}_{T}x\rangle_{\mathcal{A}} \leq M\langle x,x\rangle_{\mathcal{A}}.
	\end{equation*}
	Therefore $\{T_{\omega} \}_{\omega\in\Omega}$ is an integral operator frame for $End_{\mathcal{A}}^{\ast}(\mathcal{H})$. \\
	Conversely, let $\{T_{\omega}\}_{\omega\in\Omega}$ be an integral operator frame for $End_{\mathcal{A}}^{\ast}(\mathcal{H})$ with bounds $A$ and $B$.\\
	Then, 
	\begin{equation}\label {2.4}
	A\langle x,x\rangle_{\mathcal{A}} \leq \int_{\Omega}\langle T_{w}x,T_{w}x\rangle_{\mathcal{A}} d\mu(w) \leq B\langle x,x\rangle_{\mathcal{A}},  \qquad x\in \mathcal{H}.
	\end{equation}
	Since $0 \leq \langle x,x\rangle_{\mathcal{A}}$, for all $x\in \mathcal{H} $, then we can take the norme in the left, middle and right termes of the above inequality \eqref{2.4}.\\
	Thus,
	\begin{equation*}
	\|A\langle x,x\rangle_{\mathcal{A}}\| \leq \| \int_{\Omega}\langle T_{\omega}x,T_{\omega}x\rangle_{\mathcal{A}} d\mu(w)\| \leq \|B\langle x,x\rangle_{\mathcal{A}}\|,  \qquad x\in \mathcal{H}.
	\end{equation*}
	finally we get,
	\begin{equation*}
	A\|x\|^{2} \leq \| \int_{\Omega}\langle T_{\omega}x,T_{\omega}x\rangle_{\mathcal{A}} d\mu(w)\| \leq B\|x\|^{2}, \qquad x\in \mathcal{H}.
	\end{equation*}
	Which ends the proof.
\end{proof}
\begin{definition}
let $T=\{T_{\omega}\}_{\omega \in \Omega}$ and $\Lambda=\{\Lambda_{\omega}\}_{\omega \in \Omega}$ be two integral operator frame for $End_{\mathcal{A}}^{\ast}(\mathcal{H})$. The integral operator frame $\Lambda$ is called a dual for $T$ if for all $x\in \mathcal{H}$ we have,
\begin{equation}
x=\int_{\Omega}T^{\ast}_{\omega}\Lambda_{\omega}x d\mu(w) \qquad 
\end{equation}  
\end{definition}
\begin{example}
Let $\mathcal{H}$ be the Hilbert space defined by:\\
$\mathcal{H}=\left\{ A=\left( 
\begin{array}{ccc}
a & 0  \\ 
0 & b 
\end{array}%
\right) \text{ / }a,b\in 
\mathbb{C}
\right\} $, \\
and let $\{T_{w}\}_{w\in \left[ 0,1\right]}$ the family defined in Example 2.2. It's clear that $T_{w}=T_{w}^{\ast}$.

For all $w \in \left[ 0,1\right] $, we define :
\begin{align*}
\Lambda_{\omega} : \mathcal{H} &\longrightarrow \mathcal{H}\\
\left(   \begin{array}{ccc}
a & 0  \\ 
0 & b 
\end{array}
\right)&\longrightarrow \left(   \begin{array}{ccc}
3wa & 0  \\ 
0 & 2\sqrt{3}wb 
\end{array}
\right)
\end{align*}
It is clear that the family $\{\Lambda_{w}\}_{w\in \left[ 0,1\right]}$ is an integral operator frameon $\mathcal{H}$.\\
Moreever, we have :
\[
\begin{array}{ccc}
\langle \Lambda_{w}A, \Lambda_{w}A\rangle = \left( 
\begin{array}{ccc}
9w^{2}|a|^{2} & 0  \\ 
0 & 12w^{2}|b|^{2} 
\end{array}%
\right) %
\end{array}%
\]\\

Hence,
\[
\begin{array}{ccc}
\int_{\Omega}\langle \Lambda_{w}A, \Lambda_{w}A\rangle d\lambda(w) =\int_{[0,1]}\left( 
\begin{array}{ccc}
9w^{2}|a|^{2} & 0  \\ 
0 & 12w^{2}|b|^{2} 
\end{array}%
\right)d\lambda=\left( 
\begin{array}{ccc}
3|a|^{2} & 0  \\ 
0 & 4|b|^{2} 
\end{array}%
\right) %
\end{array}%
\]\\

So, 
\begin{equation*}
3 \langle A,A\rangle\leq \int_{\Omega}\langle \Lambda_{w}A, \Lambda_{w}A\rangle d\lambda(w) \leq 4 \langle A,A\rangle.
\end{equation*}
Which shows that  $\{\Lambda_{w}\}_{w\in \left[ 0,1\right]}$ is a integral operator frame for $\mathcal{H}$.\\
moreover, we have for $A\in \mathcal{H}$,
\begin{equation*}
\int_{[0,1]}T_{w}^{\ast}\Lambda_{w} (A)d\lambda (w)=\int_{[0,1]}T_{w}\left( 
\begin{array}{ccc}
3wa & 0  \\ 
0 & 2\sqrt{3}wb 
\end{array}%
\right)d\lambda(\omega)=\int_{[0,1]}\left( 
\begin{array}{ccc}
3w^{2}a & 0  \\ 
0 & 3w^{2}b 
\end{array}%
\right)d\lambda(\omega)=\left( 
\begin{array}{ccc}
a & 0  \\ 
0 & b 
\end{array}%
\right)
\end{equation*}
Which shows that  $\{\Lambda_{w}\}_{w\in \left[ 0,1\right]}$ is a dual integral operator frames for $\{T_{w}\}_{w\in \left[ 0,1\right]}$.
\end{example}	

\begin{theorem}
	Every integral operator frame for $End_{\mathcal{A}}^{\ast}(\mathcal{H})$ has a dual integral operator frame.
\end{theorem}
\begin{proof}
Let $T=\{T_{\omega}\}_{\omega \in \Omega}$ be an integral operator frame for $End_{\mathcal{A}}^{\ast}(\mathcal{H})$ with bounds $A$ and $B$.\\
For all $x\in \mathcal{H}$, we have,
\begin{equation}
x=S_{T}S^{-1}_{T}x=\int_{\Omega}T^{\ast}_{\omega}T_{\omega}S^{-1}_{T}xd\mu(\omega)
\end{equation}
it remains to show that $\Lambda=\{T_{\omega}S^{-1}_{T}\}_{\omega \in \Omega}$ is an integral operator frame for $End_{\mathcal{A}}^{\ast}(\mathcal{H})$. Indeed,
from \eqref{eq3}, we have,
\begin{equation*}
A\langle S^{-1}_{T}x,S^{-1}_{T}x\rangle_{\mathcal{A}} \leq\int_{\Omega}\langle T_{\omega}S^{-1}_{T}x,T_{\omega}S^{-1}_{T}x\rangle_{\mathcal{A}} d\mu(\omega)\leq B\langle S^{-1}_{T}x,S^{-1}_{T}x\rangle_{\mathcal{A}} , \quad x\in\mathcal{H}.
\end{equation*}
Since $S^{-1}_{T}$ is a surjectif and sef-adjoint operator, then there exists (from Lemma \ref{sb}) $m>0$, such that,
\begin{equation*}
Am\langle x,x\rangle_{\mathcal{A}} \leq\int_{\Omega}\langle T_{\omega}S^{-1}_{T}x,T_{\omega}S^{-1}_{T}x\rangle_{\mathcal{A}} d\mu(\omega)\leq B\|S^{-1}_{T}\|^{-2}\langle x,x\rangle_{\mathcal{A}} , \quad x\in\mathcal{H}.
\end{equation*}
which shows that $\Lambda=\{T_{\omega}S^{-1}_{T}\}_{\omega \in \Omega}$ is a dual integral operator frame for  $End_{\mathcal{A}}^{\ast}(\mathcal{H})$
\end{proof}

\begin{definition}\label{d3}
	A family of operatos $\{T_{\omega}\}_{\omega \in \Omega} $ in $End_{\mathcal{A}}^{\ast}(\mathcal{H})$ is said to be independent if the following condition is satisfied :
	\begin{equation*}
	\int_{\Omega}T^{\ast}_{\omega}x_{\omega}d\mu(\omega)=0, \quad \{x_{w}\}_{\omega \in \Omega} \in l^{2}(\Omega,\mathcal{H}) \Longrightarrow x_{\omega}=0 \quad \mu pp. \quad for\; all \quad \omega \in \Omega.
	\end{equation*}
\end{definition}
\begin{theorem}\label{t3}
	Let $T=\{T_{\omega}\}_{\omega \in \Omega} $ be an integral operator Bessel family in $End_{\mathcal{A}}^{\ast}(\mathcal{H})$. Then the following statements are true  :\begin{itemize}
		\item [1]- $T=\{T_{\omega}\}_{\omega \in \Omega} $ is an integral operator frame for $End_{\mathcal{A}}^{\ast}(\mathcal{H})$ if and only if $R_{T}$ is below bounded.
		\item [2]- $T=\{T_{\omega}\}_{\omega \in \Omega} $ is an independent integral operator frame for $End_{\mathcal{A}}^{\ast}(\mathcal{H})$ if and only if $R_{T}$ is invertible.
	\end{itemize}
\end{theorem}
\begin{proof}
(1): obvious.\\
	For (2), let $T=\{T_{\omega}\}_{\omega \in \Omega} $ be an independent integral operator frame for $End_{\mathcal{A}}^{\ast}(\mathcal{H})$.\\
	For all $\{x_{\omega}\}_{\omega \in \Omega} \in l^{2}(\Omega,\mathcal{H})$, we have :
	\begin{equation*}
	R^{\ast}_{T}(\{x_{\omega}\}_{\omega \in \Omega})=\int_{\Omega}T^{\ast}_{\omega}x_{\omega}d\mu(\omega)=0 \Longrightarrow \{x_{\omega}\}_{\omega \in \Omega}=0 \qquad for\textcolor{white}{.} all \quad \omega \in \Omega.
	\end{equation*}
	This prove that $R^{\ast}_{T}$ is injective.\\
	So, 
	\begin{equation*}
	(\mathcal{R}(R_{T}))^{\bot}=Ker(R^{\ast}_{T})=\{0\}.
	\end{equation*}
	This shows that $\mathcal{R}(R_{T})$ is dense in $H$.\\
	By (1), we know that $R_{T}$ is below bounded and $\mathcal{R}(R_{T})$ is closed, hence $R_{T}$ is invertible.\\
	For the converse, if $R_{T}$ is invertible, then $R_{T}$ is below bounded.\\
	Thus $T=\{T_{\omega}\}_{\omega \in \Omega} $ is an integral operator frame.\\
	Now, suppose that $T=\{T_{\omega}\}_{\omega \in \Omega} \subset End_{\mathcal{A}}^{\ast}(\mathcal{H}) $ is not independent, then there exists $0\neq \{\eta_{w}\}_{\omega \in \Omega} \in l^{2}(\Omega,\mathcal{H})$ such that,	\begin{equation}\label{eq1}
	\int_{\Omega}T^{\ast}_{\omega}\eta_{\omega}d\mu(\omega)=0
	\end{equation}
	Since $R_{T}$ is surjective then there exists $x\in \mathcal{H}$ such that,
	\begin{equation*}
	R_{T}x=\{T_{\omega}x\}_{\omega \in \Omega}=\{\eta_{\omega} \}_{\omega \in \Omega}
	\end{equation*}
	So
	\begin{equation*}
	\eta_{\omega}=T_{\omega}x \qquad \forall \omega \in \Omega
	\end{equation*}
	On one hand we have, 
	\begin{align*}
	\langle R_{T}x,\{\eta_{\omega} \}_{\omega \in \Omega}\rangle_{l^{2}(\Omega,\mathcal{H})}&=\langle \{T_{\omega}x\}_{\omega \in \Omega},\{\eta_{\omega} \}_{\omega \in \Omega}\rangle_{l^{2}(\Omega,\mathcal{H})}\\
	&=\int_{\Omega}\langle T_{\omega}x,\eta_{\omega}\rangle_{\mathcal{A}}d\mu(\omega)\\
	&=\int_{\Omega}\langle x,T^{\ast}_{\omega}\eta_{\omega}\rangle_{\mathcal{A}}d\mu(\omega)\\
	&= \langle x,\int_{\Omega}T^{\ast}_{\omega}\eta_{\omega}d\mu(\omega)\rangle_{\mathcal{A}}=0
	\end{align*}
	One other hand, we have
	\begin{align*}
	\langle R_{T}x,\{\eta_{\omega} \}_{\omega \in \Omega}\rangle_{l^{2}(\Omega,\mathcal{H})}&=\langle x,R^{\ast}_{T}(\{\eta_{\omega} \}_{\omega \in \Omega})\rangle_{\mathcal{A}}\\
	&=\langle x,\int_{\Omega}T^{\ast}_{\omega}\eta_{\omega}d\mu(\omega)\rangle_{\mathcal{A}}\\
	&=\langle x,\int_{\Omega}T^{\ast}_{\omega}T_{\omega}xd\mu(\omega)\rangle_{\mathcal{A}}\\
	&=\int_{\Omega}\langle x,T^{\ast}_{\omega}T_{\omega}x\rangle_{\mathcal{A}}d\mu(\omega)\\
	&=\int_{\Omega}\langle T_{\omega}x,T_{\omega}x\rangle_{\mathcal{A}}d\mu(\omega)\\
	&=\|\eta_{\omega}\|_{l^{2}(\Omega,\mathcal{H})}\neq 0
	\end{align*}
	It is a contradiction.\\
	This show that $T=\{T_{\omega}\}_{\omega \in \Omega} $ is independent integral operator frame for $End_{\mathcal{A}}^{\ast}(\mathcal{H})$.
\end{proof}

	\section{Perturbation of Integral Operator Frames}

\begin{theorem}
	Let $\{T_{\omega}\}_{\omega \in \Omega}$ be an integral operator frame for $End_{\mathcal{A}}^{\ast}(\mathcal{H})$ with frames bounds $A$ and $B$. Let $ K \in End_{\mathcal{A}}^{\ast}(\mathcal{H}), (K\neq 0)$, and $\{c_{\omega}\}_{\omega \in \Omega}$ any family of scalars. Then the perturbed family of operator $\{T_{\omega} + c_{\omega}K\}_{\omega \in \Omega}$ is an integral operator frames for $End_{\mathcal{A}}^{\ast}(\mathcal{H})$ if $\int_{\Omega}|c_{\omega}|^{2}d\mu(\omega) < \frac{A}{\|K\|}$.
\end{theorem}
\begin{proof}
	Let $\Gamma_{\omega}= T_{\omega} + c_{\omega}K$,  for all $\omega \in \Omega$. Then for all $x\in \mathcal{H}$, we have,
	\begin{align*}
	\int_{\Omega}\langle T_{\omega}x-\Gamma_{\omega}x,T_{\omega}x-\Gamma_{\omega}x\rangle_{\mathcal{A}} d\mu(\omega)&=\int_{\Omega}\langle c_{\omega}Kx,c_{\omega}Kx\rangle_{\mathcal{A}} d\mu(\omega)\\
	&\leq \int_{\Omega}|c_{\omega}|^{2}\|K\|^{2}\langle x,x\rangle_{\mathcal{A}} d\mu(\omega)\\
	&\leq R\langle x,x\rangle_{\mathcal{A}}. \quad where \quad  R=\int_{\Omega}|c_{\omega}|^{2}\|K\|^{2}d\mu(\omega).
	\end{align*}
	On one hand, for all $ x\in \mathcal{H}$, we have, 
	\begin{align*}
	(\int_{\Omega}\langle (T_{\omega} + c_{\omega}K)x,(T_{\omega} + c_{\omega}K)x\rangle_{\mathcal{A}} d\mu(\omega))^{\frac{1}{2}}&=\|(T_{\omega} + c_{\omega}K)x\|_{l^{2}(\Omega,H)}\\
	&\leq \|T_{\omega}x\|_{l^{2}(\Omega,H)} + \| c_{\omega}Kx\|_{l^{2}(\Omega,H)}\\
	&= (\int_{\Omega}\langle T_{\omega}x,T_{\omega}x\rangle_{\mathcal{A}} d\mu(\omega))^{\frac{1}{2}}  \\
	& \quad + (\int_{\Omega}\langle (c_{\omega}K)x,(c_{\omega}K)x\rangle_{\mathcal{A}} d\mu(\omega))^{\frac{1}{2}}\\
	&\leq \sqrt{B}(\langle x,x\rangle_{\mathcal{A}})^{\frac{1}{2}} +  \sqrt{R}(\langle x,x\rangle_{\mathcal{A}})^{\frac{1}{2}}\\
	&\leq (\sqrt{B} + \sqrt{R})(\langle x,x\rangle_{\mathcal{A}})^{\frac{1}{2}}.
	\end{align*}
	Then 
	\begin{equation}\label{123}
	\int_{\Omega}\langle (T_{\omega} + c_{\omega}K)x,(T_{\omega} + c_{\omega}K)x\rangle_{\mathcal{A}} d\mu(\omega)\leq (\sqrt{B} + \sqrt{R})^{2}\langle x,x\rangle_{\mathcal{A}}.
	\end{equation}
	On the other hand, for all $x\in H$, we have,
	\begin{align*}
	(\int_{\Omega}\langle (T_{\omega} + c_{\omega}K)x,(T_{\omega} + c_{\omega}K)x\rangle_{\mathcal{A}} d\mu(\omega))^{\frac{1}{2}}&=\|(T_{\omega} + c_{\omega}K)x\|_{l^{2}(\Omega,\mathcal{H})}\\
	&\geq \|T_{\omega}x\|_{l^{2}(\Omega,\mathcal{H})} - \| c_{\omega}Kx\|_{l^{2}(\Omega,\mathcal{H})}\\
	&\geq (\int_{\Omega}\langle T_{\omega}x,T_{\omega}x\rangle_{\mathcal{A}} d\mu(\omega))^{\frac{1}{2}} \\
	& \textcolor{white}{.....} - (\int_{\Omega}\langle c_{\omega}Kx,c_{\omega}Kx\rangle_{\mathcal{A}} d\mu(\omega))^{\frac{1}{2}}\\
	&\geq \sqrt{A}(\langle x,x\rangle_{\mathcal{A}})^{\frac{1}{2}} -  \sqrt{R}(\langle x,x\rangle_{\mathcal{A}})^{\frac{1}{2}}\\
	&\geq (\sqrt{A} -  \sqrt{R})(\langle x,x\rangle_{\mathcal{A}})^{\frac{1}{2}}.
	\end{align*}
	So,
	\begin{equation}\label{124}
	\int_{\Omega}\langle (T_{\omega} + c_{\omega}K)x,(T_{\omega} + c_{\omega}K)x\rangle_{\mathcal{A}} d\mu(\omega) \geq  (\sqrt{A} -  \sqrt{R})^{2}\langle x,x\rangle_{\mathcal{A}} 
	\end{equation}
	From \eqref{123} and \eqref{124} we conclude that  $\{T_{\omega} + c_{\omega}K\}_{\omega \in \Omega}$ is an integral operator frame for $End_{\mathcal{A}}^{\ast}(\mathcal{H})$ if $R<A$, that is , if :
	\begin{equation*}
	\int_{\Omega}|c_{\omega}|^{2}d\mu(\omega) < \frac{A}{\|K\|}.
	\end{equation*}
	
\end{proof}
\begin{theorem}
	Let $\{T_{\omega}\}_{\omega \in \Omega} \subset End_{\mathcal{A}}^{\ast}(\mathcal{H})$ be an integral operator frame. Let $\{\Lambda _{\omega}\}_{\omega \in \Omega}$ be any family of operators on $End_{\mathcal{A}}^{\ast}(\mathcal{H})$, and $\{a_{\omega}\}_{\omega \in \Omega},\, \{b_{\omega}\}_{\omega \in \Omega} \subset \mathbb{R}$ be two positively  confined sequences. If there exist a constants $\alpha , \beta$  with $0\leq \alpha , \beta<\frac{1}{2}$ such that,
	\begin{align}\label{1000}
	\int_{\Omega}\langle a_{\omega}T_{\omega}x - b_{\omega}\Lambda _{\omega}x,a_{\omega}T_{\omega}x - b_{\omega}\Lambda _{\omega}x\rangle d\mu(\omega)&\leq \alpha\int_{\Omega}\langle a_{\omega}T_{\omega}x,a_{\omega}T_{\omega}x\rangle d\mu(\omega) \\
	&\textcolor{white}{.....} + \beta\int_{\Omega}\langle b_{\omega}\Lambda _{\omega}x,b_{\omega}\Lambda _{\omega}x\rangle d\mu(\omega).
	\end{align}
	Then $\{\Lambda _{\omega}\}_{\omega \in \Omega}$ is an integral operator frame for $End_{\mathcal{A}}^{\ast}(\mathcal{H})$.
\end{theorem}
\begin{proof}
	Suppose \eqref{1000} holds for some conditions of theorem.\\
	Then for all $x\in H$ we have,
	\begin{align*}
	\int_{\Omega}\langle b_{\omega}\Lambda _{\omega}x,b_{\omega}\Lambda _{\omega}x\rangle d\mu(\omega)&\leq 2(\int_{\Omega}\langle a_{\omega}T_{\omega}x,a_{\omega}T_{\omega}x\rangle d\mu(\omega)\\
	&\textcolor{white}{.....} +\int_{\Omega}\langle a_{\omega}T_{\omega}x - b_{\omega}\Lambda _{\omega}x,a_{\omega}T_{\omega}x - b_{\omega}\Lambda _{\omega}x\rangle d\mu(\omega) ) \\
	&\leq 2(\int_{\Omega}\langle a_{\omega}T_{\omega}x,a_{\omega}T_{\omega}x\rangle d\mu(\omega) +\alpha\int_{\Omega}\langle a_{\omega}T_{\omega}x,a_{\omega}T_{\omega}x\rangle d\mu(\omega)\\
	&\textcolor{white}{.....} + \beta\int_{\Omega}\langle b_{\omega}\Lambda _{\omega}x, b_{\omega}\Lambda _{\omega}x\rangle d\mu(\omega)).
	\end{align*}
	So,
	\begin{equation*}
	(1-2\beta)\int_{\Omega}\langle b_{\omega}\Lambda _{\omega}x, b_{\omega}\Lambda _{\omega}x\rangle d\mu(\omega))\leq 2(1+\alpha)\int_{\Omega}\langle a_{\omega}T_{\omega}x,a_{\omega}T_{\omega}x\rangle d\mu(\omega).
	\end{equation*}
	This give,
	\begin{equation*}
	(1-2\beta)[\underset{\omega \in \Omega}{\inf}( b_{\omega})]^{2}\int_{\Omega}\langle \Lambda _{\omega}x,\Lambda _{\omega}x\rangle d\mu(\omega)\leq 2(1+\alpha)[\underset{\omega \in \Omega}{\sup}( a_{\omega})]^{2}\int_{\Omega}\langle T_{\omega}x,T_{\omega}x\rangle d\mu(\omega).
	\end{equation*}
	Thus,
	\begin{equation}\label{101}
	\int_{\Omega}\langle \Lambda _{\omega}x,\Lambda _{\omega}x\rangle d\mu(\omega)\leq \frac{ 2(1+\alpha)[\underset{\omega \in \Omega}{\sup}( a_{\omega})]^{2}}{(1-2\beta)[\underset{\omega \in \Omega}{\inf}( b_{\omega})]^{2}}\int_{\Omega}\langle T_{\omega}x,T_{\omega}x\rangle d\mu(\omega).
	\end{equation}
	Also, we have,
	\begin{align*}
	\int_{\Omega}\langle a_{\omega}T_{\omega}x,a_{\omega}T_{\omega}x\rangle d\mu(\omega)&\leq 2(\int_{\Omega}\|a_{\omega}T_{\omega}x - b_{\omega}\Lambda _{\omega}x\|^{2}d\mu(\omega) + \int_{\Omega}\|b_{\omega}\Lambda _{\omega}x\|^{2}d\mu(\omega))\\
	&\leq 2(\alpha\int_{\Omega}\|a_{\omega}T_{\omega}x\|^{2}d\mu(\omega) + \beta\int_{\Omega}\|b_{\omega}\Lambda _{\omega}x\|^{2}d\mu(\omega)+\int_{\Omega}\|b_{\omega}\Lambda _{\omega}x\|^{2}d\mu(\omega)).
	\end{align*}
	Therefore,
	\begin{equation*}
	(1-2\alpha)[\underset{\omega \in \Omega}{\inf}( a_{\omega})]^{2}\int_{\Omega}\|T_{\omega}x\|^{2}d\mu(\omega)\leq 2(1+\beta)[\underset{\omega \in \Omega}{\sup}( b_{\omega})]^{2}\int_{\Omega}\|\Lambda _{\omega}x\|^{2}d\mu(\omega).
	\end{equation*}
	This give:
	\begin{equation}\label{102}
	\frac{(1-2\alpha)[\underset{\omega \in \Omega}{\inf}( a_{\omega})]^{2}}{2(1+\beta)[\underset{\omega \in \Omega}{\sup}( b_{\omega})]^{2}}\int_{\Omega}\|T_{\omega}x\|^{2}d\mu(\omega)\leq \int_{\Omega}\|\Lambda _{\omega}x\|^{2}d\mu(\omega)
	\end{equation}
	From \eqref{101} and \eqref{102} we conclude,
	\begin{equation*}
	\frac{(1-2\alpha)[\underset{\omega \in \Omega}{\inf}( a_{\omega})]^{2}}{2(1+\beta)[\underset{\omega \in \Omega}{\sup}(
		 b_{\omega})]^{2}}\int_{\Omega}\|T_{\omega}x\|^{2}d\mu(\omega)\leq \int_{\Omega}\langle \Lambda _{\omega}x,\Lambda _{\omega}x\rangle \mu(\omega)\leq \frac{ 2(1+\alpha)[\underset{\omega \in \Omega}{\sup}( a_{\omega})]^{2}}{(1-2\beta)[\underset{\omega \in \Omega}{\inf}( b_{\omega})]^{2}}\int_{\Omega}\|T_{\omega}x\|^{2}d\mu(\omega).
	\end{equation*}
	which ends the proof.
\end{proof}

\end{document}